\documentclass[10pt,a4paper]{article}
\usepackage[utf8]{inputenc}
\usepackage[english]{babel}
\usepackage[T1]{fontenc}
\usepackage{amsmath}
\usepackage{amsfonts}
\usepackage{amssymb}
\usepackage[left=3cm,right=3cm,top=4cm,bottom=3cm]{geometry}
\usepackage{graphicx}
\usepackage{caption} 
\usepackage{ulem}
\usepackage{comment}
\usepackage{color}

\usepackage{charter}
\usepackage{layout}

\usepackage{graphicx}
\usepackage[all]{xy}

\usepackage{amsthm}
\usepackage{amsmath}
\usepackage{amssymb}
\usepackage{mathrsfs}

\newtheorem{theo}{Theorem}
\newtheorem{proposition}[theo]{Proposition}
\newtheorem{lemme}[theo]{Lemma}
\newtheorem{corollaire}[theo]{Corollary}
\newtheorem*{definition}{Definition}
\newtheorem*{remarque}{Remark}

\newcommand{\loc}{\mathrm{loc}}

\newcommand{\dd}{\mathrm{d}}
\newcommand{\pp}{\text{.}}
\newcommand{\vv}{\text{,}}
\newcommand{\dt}{\partial_t}
\newcommand{\bb}{\mathbf}

\DeclareMathOperator{\trace}{Tr}

\DeclareMathOperator{\supp}{supp}

\begin{document}
\begin{center}{\Large{\textbf{Best exponential decay rate of energy for the vectorial damped wave equation }}}

\end{center}
\bigskip
\begin{center}
Guillaume Klein
\end{center}

\bigskip
\bigskip

\begin{abstract}
The energy of solutions of the scalar damped wave equation decays uniformly exponentially fast when the geometric control condition is satisfied. A theorem of Lebeau \cite{leb93} gives an expression of this exponential decay rate in terms of the average value of the damping terms along geodesics and of the spectrum of the infinitesimal generator of the equation. The aim of this text is to generalize this result in the setting of a vectorial damped wave equation on a Riemannian manifold with no boundary. We obtain an expression analogous to Lebeau's one but new phenomena like high frequency overdamping arise in comparison to the scalar setting.  We also prove a necessary and sufficient condition for the strong stabilization of the vectorial wave equation. 
\end{abstract}

\section{Introduction}

Let $(M,g)$ be a smooth, connected, compact Riemannian manifold without boundary of dimension $d$. Let $\Delta$  be the Laplace-Beltrami's operator on $M$ for the metric $g$ and let $a$ be a smooth function from $M$ to $\mathscr{H}_n^+(\mathbf C)$, the space of  positive-semidefinite hermitian matrices of dimension $n$. We are interested in the following system of equations
\begin{equation}\label{dampedwaveequation}
\left\lbrace
\begin{array}{l}
 (\partial_t^2 -\Delta +2a(x)\partial_t)u=0 \; \text{ in } \; \mathcal D'(\mathbf R\times M)^n \\

u_{|t=0}=u_0\in H^1(M)^n\; \text{ and } \; \partial_tu_{|t=0}=u_1\in L^2(M)^n\pp
\end{array}
\right.
\end{equation}
Let $H=H^1(M)^n\oplus L^2(M)^n$ and define on $H$ the unbounded operator
\[
A_a=\begin{pmatrix}
0 & \mathrm{Id}_n \\
\Delta & -2a
\end{pmatrix} \text{ of domain } D(A_a)=H^2(M)^n\oplus H^1(M)^n\pp
\]
By application of Hille-Yosida's theorem to $A_a$ the system \eqref{dampedwaveequation} has a unique solution in the space $ C^0(\mathbf R,H^1(M)^n)\cap C^1(\mathbf R, L^2(M)^n)$, from now on we will identify $H$ with the space of solutions of \eqref{dampedwaveequation}. The euclidean norm on $\mathbf R^n $ or $\mathbf C^n$ will be written  $|\cdot|$ and we will write $\langle\cdot ,\cdot\rangle_{\mathcal H}$ the inner product of an Hilbert space $\mathcal H$ or simply $\langle\cdot,\cdot\rangle$ when there is no possible confusion. Let us define $E(u,t)$, the energy of a solution $u$ at time $t$, by the formula 
\[
E(u,t)=\frac{1}{2}\int_{M}|\partial_t u(t,x)|^2 + |\nabla u(t,x)|^2 \dd x
\]
where $|\nabla u(t,x)|^2=g_x(\nabla u(t,x),\nabla u(t,x))$. We then have the relation
\begin{equation}\label{formuleenergie}
E(u,T)=E(u,0)-\int_0^T\int_{M}\big\langle 2a(x)\partial_t u(t,x),\partial_t u(t,x)\big\rangle_{\mathbf C^n}\dd x \dd t \pp
\end{equation}
The energy is thus a non-increasing function of time. We are interested in the problem of stabilization of the wave equation; that is, determining the long time behavior of the energy. This has been well studied in the scalar setting ($n=1$) but not so much in the vectorial setting ($n>1$). Nevertheless, the stabilization of the vectorial wave equation is an interesting and naturally occurring problem. The aim of this article is to adapt and prove some classical results of scalar stabilization to the vectorial case, we will also highlight the main differences between the two settings. The most basic result about stabilization of the wave equation is probably the following.

\begin{theo}\label{stabfaible}
The following conditions are equivalent.
\begin{description}
\item[(i)] $\forall u \in H \; \; \displaystyle \lim_{t\to \infty}E(u,t)=0$
\item[(ii)] The only eigenvalue of $A_a$ on the imaginary axis is $0$.
\end{description}
Moreover, if $a$ is definite positive at one point (and thus on an open set) then the two conditions above are satisfied.
\end{theo}
The condition \textbf{(i)} is called weak stabilisation of the damped wave equation. For a succinct proof of this result see the introduction of \cite{leb93}, for a more detailed proof in a simpler setting see Theorem 4.2 of \cite{buge01}. Note that when $n=1$ (\textit{ie} in the scalar case) there is a more satisfactory result stating that the condition \textbf{(i)} is equivalent to $a\neq 0$.

\begin{theo}\label{stabforte}
The following conditions are equivalent.

\textbf{\textup{(i)}} There is weak stabilisation and for every maximal geodesic $s\in \mathbf R \mapsto x_s$ of $M$ we have
\[
\bigcap_{s\in \mathbf R} \ker(a(x_s))=\{0\}.\label{GCC}\tag{GCC}
\]

\textbf{\textup{(ii)}} There exists two constants $C,\beta >0$ such that for all $u\in H$ and for every time $t$  
\[
E(u,t)\leq C e^{-\beta t} E(u,0)\pp
\]

\end{theo}

The condition on the intersections of the kernels of $a(x_t)$ is called the Geometric Control Condition (GCC) and the condition \textbf{(ii)} is called strong stabilisation of the damped wave equation. For $n=1$ this theorem has been proved in the more general setting of a riemannian manifold with boundary by Bardos, Lebeau, Rauch and Taylor (\cite{rata74} and \cite{blr92}). Note that, when $n=1$, the weak stabilization hypothesis is not needed because it is a consequence of the geometric control condition. However when $n>1$ the geometric condition alone does not imply strong or even weak stabilization as we shall see later, so this hypothesis is necessary. It is still an open problem to find a purely geometric condition equivalent to strong stabilization of the vectorial wave equation. To my knowledge Theorem \ref{stabforte} has not been proved in the existent literature, but it seems that it was already known by people well acquainted with the field. We will get a proof of Theorem \ref{stabforte} as a corollary of Theorem \ref{bigtheoreme}.

\begin{definition}
We denote the best exponential decay rate of the energy by $\alpha$ defined as follow : 
\[\alpha=\sup \{\beta \in \mathbf R : \exists C>0 , \forall u\in H, \forall T>0 , E(u,T)\leq C e^{-\beta T}E(u,0)\}\pp\]
\end{definition}

The main result of this article is Theorem \ref{bigtheoreme}, its aim is to express $\alpha$ as the minimum of two quantities. The first quantity depends on the spectrum of $A_a$ and the second one depends on a differential equation described by the values of $a$ along geodesics. However we still need to define a few things before being able to state Theorem \ref{bigtheoreme}.

\paragraph*{} It is well known that $\mathrm{sp}(A_a)$, the spectrum of $A_a$, is discrete and solely contains eigenvalues $\lambda_j$ satisfying $\mathfrak{Re}(\lambda_j)\in [-2\sup_{x\in M}\|a(x)\|_2;0]$ and $|\lambda_j|\to \infty$. This comes from the fact that $D(A_a)$ is compactly embedded in $H$ and that, for $\mathfrak{Re }(\lambda)\notin [-2\sup_{x\in M}\|a(x)\|_2;0] $, the operator $(A_a-\lambda \mathrm{Id})$ is bijective from $D(A_a) $ to $H$ and has a continuous inverse. Moreover the spectrum of $A_a$ is invariant by complex conjugation. We will denote by $E_{\lambda_j}$ the generalized eigenvector subspace of $A_a$ associated with $\lambda_j$, this subspace is defined as
\[
E_{\lambda_j}=\left\{u\in D(A_a) : \exists k \in \mathbf{N} , (A_a-\lambda_j)^ku=0 \right\}
\]
and is of finite dimension. We next define the following quantities.
\begin{equation}
D(R)=\sup \{ \mathfrak{Re}(\lambda_j) : \lambda_j \in \mathrm{sp}(A_a), |\lambda_j| > R \}\vv \;\; D_0=\lim_{R\to 0^+} D(R) \; \text{ and } \; D_\infty=\lim_{R\to\infty}D(R)\pp
\end{equation}
These quantities are all non negative and for every $R>0$ we have $D_0\geq D(R)\geq D_\infty$. The quantity $D_0$ is sometime called the spectral abscissa of $A_a$.

\paragraph*{}Since $M$ is a Riemannian manifold there is a natural isometry between $T_xM$ and $T^*_xM$ \textit{via} the scalar product $g_x$. The scalar product defined on $T^*_xM$ by this isometry is called $g^x$ and if $\xi\in T_x^*M$ we will write $|\xi|_g$ for $\sqrt{g^x(\xi,\xi)}$.   Let us call $S^*M$ the  cotangent sphere bundle of $M$, that is, the subset $\{(x,\xi)\in T^*M : |\xi|_g=1/2\}$ of $T^*M$. We call $\phi$ the geodesic flow on $S^*M$ and recall that it corresponds to the Hamiltonian flow generated by $|\xi|_g^2$. In everything that follows $(x_0;\xi_0)$ will denote a point of $S^*M$ and we will write $(x_t,\xi_t)$ for $\phi_t(x_0,\xi_0)$. We now introduce the function $G^+_t : S^* M \to \mathscr M_n(\mathbf C)$ where $t$ is a real number. It is defined as the solution of the differential equation
\begin{equation}\label{eq:equationG}
\left\lbrace
\begin{array}{l}
G^+_0(x_0,\xi_0)=\mathrm{Id}_n \\
\partial_t G^+_t(x_0,\xi_0)=-a(x_t)G^+_t(x_0,\xi_0)\pp
\end{array}
\right.
\end{equation}
We shall see later that $G_t^+$ is a cocycle map, this means that it satisfy the relation $G_{s+t}^+(x,\xi)=G^+_t(\phi_s(x,\xi))G_s^+(x,\xi)$. In the scalar-like case where $a(x)$ is a diagonal matrix everywhere the matrix $G_t^+$ is simply described by the formula  
\begin{equation}\label{scalarG}
G^+_t(x_0,\xi_0)=\exp\left(-\int_0^t a(x_s) \dd s \right)\pp\end{equation}
As we will see, the fact that this formula is no longer true in the general setting is the main reason why new phenomena arise in comparison to the scalar case (see for example Proposition \ref{prop4}). Let us define for every $t>0$ the quantities 
\begin{equation}
C(t)\overset{\mathrm{def}}{=}\frac{-1}{t}\sup_{(x_0,\xi_0)\in S^*M} \ln \left( \|G_t^+(x_0;\xi_0)\|_2 \right) \;\text{ and }\; C_\infty=\lim_{t\to \infty }C(t)\vv
\end{equation}
we will see later that this limit does exist.
In the scalar case one also have the simpler formula
\begin{equation}\label{eq:formuleGscalaire}
C(t)=\frac{1}{t} \inf_{(x_0,\xi_0)\in S^*M} \int_0^t  a(x_s)\dd s\pp
\end{equation}
There is a similar but more complex formula in the general case. Denote by $y_t$ a vector of $\mathbf C^n$ of euclidean norm $1$ such that 
\begin{equation}
G^+_{t}(x_0,\xi_0)G^+_{t}(x_0,\xi_0)^* y_t= \|G^+_{t}(x_0,\xi_0)\|_2^2y_t \pp
\end{equation}
The vector $y_t$ obviously depends on $(x_0,\xi_0)$ even though it is not explicitly written. We then have for every $t>0$
\begin{equation}\label{eq:formuleCinftyagainagain}
C(t)=\frac{1}{t} \inf_{(x_0,\xi_0)\in S^*M} \int_0^t \langle a(x_s)y_s,y_s\rangle\dd s \pp
\end{equation}
This formula is a direct consequence of Proposition \ref{prop:formulenormeG} and does not depends on the choice of $y_s$. Since $a$ is Hermitian positive semi-definite it follows from \eqref{eq:formuleCinftyagainagain} that $C(t)\geq 0$ and $C_\infty\geq 0$. Recall also that $D(0)\leq 0$ and we can finally state the main result of this article :

\begin{theo}\label{bigtheoreme}
The best exponential decay rate is given by the formula

\begin{equation}
\alpha= 2\min\{-D_0; C_\infty\}\vv
\end{equation}
moreover we have the following properties.
\begin{description}
\item[(i)]$C_\infty\leq -D_\infty$
\item[(ii)]One can have $-D_0>0$ and $C_\infty=0$.
\item[(iii)]One can have $C_\infty>0$ and $D_0=0$, but only if $n>1$.
\end{description}
\end{theo}

This result has already been proved by G. Lebeau (\cite{leb93}) for a $n=1$ on a riemannian manifold \textit{with boundary}. The novelty of this article thus comes from the fact that we are dealing with vectorial waves with a matrix damping term, this leads to the apparition of interesting new phenomena in comparison to the scalar setting (see for example section 4). The proof of Theorem \ref{bigtheoreme} stays close from the one of Lebeau and so it is pretty likely that it would extend to the case where $\partial M \neq \emptyset$ if one would be willing to adapt Corollary \ref{propagationmesure}. Let us also point out a similar result about the asymptotic behavior of the observability constant of the wave equation in Theorem 2 and Corollary 4 of \cite{hpt16}.

\begin{remarque}
We will show in the proof of Theorem \ref{stabforte} that the geometric control condition
is in fact equivalent to $C_\infty >0$. Combining this with point \textbf{(iii)} of Theorem \ref{bigtheoreme} we already see that \eqref{GCC} is not equivalent to strong stabilization when $n>1$. Moreover, using point \textbf{(i)} of theorem \eqref{bigtheoreme}, we see that when $C_\infty>0$ and $D_0=0$ we have \eqref{GCC} but weak stabilization still fails.
\end{remarque}

\begin{remarque}
Proposition \ref{lemmeinegaliteahauteenergie} and Proposition \ref{gaussianbeamenergylocalisation} show that $C_\infty$ is taking account of the energy decay of the high frequency solutions of \eqref{dampedwaveequation}. On the other hand we have $D_0\geq D_\infty$ and $-C_\infty \geq D_\infty$, so if $-D_0<C_\infty$ there exists an eigenfunction $u$ of $A_a$ such that $E(u,t)=e^{-2D_0 t}E(u,0)=e^{-\alpha t}E(u,0)$. This means that $D_0$ is taking account of the energy decay of low frequency solutions of \eqref{dampedwaveequation}.
\end{remarque}

\paragraph*{High frequency overdamping} A natural question to ask oneself is how does $\alpha$ behaves in function of the damping term $a$. Let us respectively write $\alpha(a)$, $D_0(a)$ and $C_\infty(a)$ for the quantities $\alpha$, $D_0$ and $C_\infty$ associated with a damping term $a$. An interesting fact is that the function $a\mapsto \alpha(a)$ is not monotonous, even in the simplest case. Indeed in \cite{cozu93} S. Cox and E. Zuazua showed that\footnote{Provided that $a$ is of bounded variation.}, in the case of a scalar damped wave equation on a string of length one, the decay rate is given by $\alpha(a)=-2D_0(a)$. They also calculated the spectral abscissa $D_0(a)$ in the case of a constant damping term and found $D_0(a)=-a+\mathfrak{Re}(\sqrt{a^2-\pi^2})$. This shows that increasing the constant damping term above $\pi$ actually reduces $\alpha(a)$, such a phenomenon is called ``overdamping''.

Theorem 2 of \cite{leb93} shows that for a scalar damped wave equation on a general manifold the decay rate $\alpha(a)$ is governed by $D_0(a)$ \textit{and} $C_\infty(a)$. However in that case the overdamping can only come from $D_0$ since $a\mapsto C_\infty(a)$ is obviously monotonous, sub-additive and positively homogeneous from \eqref{eq:formuleGscalaire}. In view of the previous remark it makes sens to call this phenomenon ``low frequency overdamping''.

On the other hand with the \textit{vectorial} damped wave equation the situation is different. We will show that $a\mapsto C_\infty(a)$ is neither monotonous nor sub-additive or homogeneous and thus an overdamping phenomenon can also come from the $C_\infty$ term. Once again in view of the previous remark we call this phenomenon ``high frequency overdamping''. Bellow, Figure 1 illustrates the non linear behavior of $a\mapsto C_\infty(a)$ in a specific example. To be more precise we will prove the following result.
\begin{proposition}\label{prop4}
The function $a\mapsto C_\infty(a)$ is neither homogeneous nor monotonous, more precisely it is possible to have $C_\infty(2a)<C_\infty(a)$ or $2C_\infty(a)<C_\infty(2a)$. It is also not additive, $C_\infty(a+b)$ can be strictly greater or smaller than $C_\infty(a)+C_\infty(b)$. 
\end{proposition}

\begin{figure}[ht] 
	\centering
		\includegraphics[width=1\textwidth]{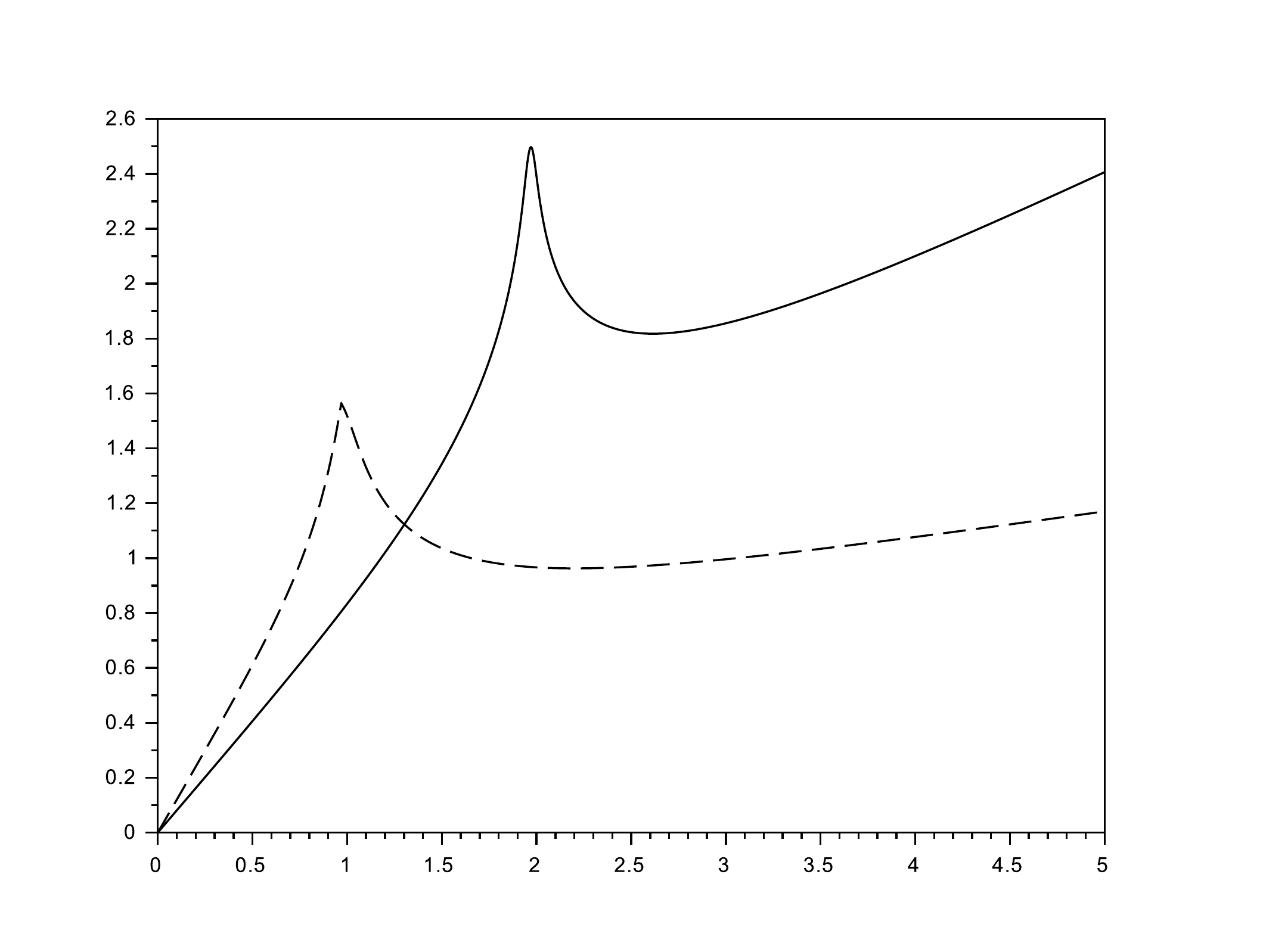}
	\caption{Plot of the function $\lambda\mapsto C_\infty(\lambda a)$ for two different damping term $a$ on $S^1$. }
	\label{fig:intro1}
\end{figure}

However it seems that $C_\infty$ still has some kind of linear behavior. Namely on $M=S^1$ and with a particular kind of damping term (see Section 4) we are able to show that
\[
\lim_{\lambda \to \infty} \frac{C_\infty(\lambda a)}{\lambda}  \;\text{ and } \; \lim_{\lambda \to 0^+} \frac{C_\infty(\lambda a)}{\lambda} 
\] both exist and are finite. This result is proved in section \ref{section4} but
it remains open to know if this is still true for any damping term on a general manifold $M$.

\paragraph*{} The remainder of this article is organized as follow.
Section \ref{section2} contains definitions and results about the propagation of the microlocal defect measures associated with a sequence of solutions of \eqref{dampedwaveequation}. These results will play an important role while bounding $\alpha$ from below. The section 3 is devoted to the proof of Theorem \ref{stabforte} and Theorem \ref{bigtheoreme}. Establishing the formula for $\alpha$ is the most difficult part, the lower bound proof makes use of Gaussian beams while for the upper bound we will use the result of section \ref{section2} conjointly with a decomposition in high and low frequencies. Eventually in the last section we study the behavior of $C_\infty$ and prove Proposition \ref{prop4}.

\section{Propagation of the microlocal deffect measure}\label{section2}

Let us work with the manifold $\mathbf R \times M$ endowed with the product metric induced by the ones of $\mathbf R$ and $M$.  We will denote by $(t,\tau,x,\xi)$ the points of $T^*(\mathbf R \times M)$, where $(t,\tau)\in T^*\bf R$ and $(x,\xi)\in T^*M$. Given a point $(x,\xi)\in T^*M$ we will write $|\xi|_g^2=g^x(\xi,\xi)$ the square of the norm of $\xi$. We moreover define $S^*(\mathbf R \times M)$ as the subset of points of $T^*(\mathbf R \times M)$ such that $\tau^2+|\xi|_g^2=1/2$ and recall that $S^*M=\{(x,\xi)\in T^*M : |\xi|^2_g=1/4\}$. We call $\phi$ the geodesic flow on $T^* M$, that is, the Hamiltonian flow generated by $|\xi|_g^2$ and $\varphi$ the Hamiltonian flow on $T^*(\mathbf R\times M)$ generated by $|\xi|_g^2-\tau^2$. In other words
\[
\varphi_s(t,\tau,x,\xi)=(t-2s\tau,\tau, \phi_s(x,\xi))\pp
\]
In everything that follows $(x_0,\xi_0)$ will denote a point of $S^*M$ and we will write $(x_t,\xi_t)=\phi_t(x_0,\xi_0)$.

\paragraph*{}Throughout this section we call $P$ the differential operator $\dt^2-\Delta $, we know that $P$ is self-adjoint on $L^2(\mathbf R \times M)^n$ and has $(|\xi|_g^2-\tau^2)\mathrm{Id}_n=p\cdot\mathrm{Id}_n$ for principal symbol, note that $p$ is a scalar valued function. If $b$ is a smooth function from $T^*(\mathbf R \times M)$ to $\mathscr M_n(\mathbf C)$ we note $\{p,b\}$ the Poisson's bracket of $p$ and $b$, it is defined as the matrix whose coefficients are the usual Poisson's bracket $\{p,b_{ij}\}$. With this definition the basic properties of Poisson's bracket are still true. Namely, we have a Leibniz's rule $\{p,bc\}=\{p,b\}c+b\{p,c\}$ and it is linked to the Hamiltonian flow of $p$ in the usual way, that is $\partial_s( b\circ \varphi_s) (\rho) = \{p,b\}(\varphi_s(\rho))$. Moreover, if $B$ is a pseudo-differential operator of order $m$ and of principal symbol $\sigma_m(B)=b$ then $[P,B]$ is a pseudo-differential operator of order $\leq m+1$ and of principal symbol $-i\{p,b\}$. Note that this is only possible because $p\cdot\mathrm{Id}$ commutes with every matrix of $\mathscr M_n(\mathbf C)$. For more details about pseudo-differential operators see \cite{hor85}.

\paragraph*{}We now recall some results about microlocal defect measures. For proofs and more details see the original article of P. Gérard \cite{ger91}.

\begin{proposition}
Let $(u_n)_n$  be a sequence of functions of $H^m_{\loc}(\mathbf R\times M)$  weakly converging to $0$. Then there exists 
\begin{description}
\item[-]a sub-sequence  $(u_{n_k})_k$,
\item[-]a positive Radon measure $\nu$ on $S^*(\mathbf R \times M)$,
\item[-]a matrix $M$ of $\nu$-integrable functions on $S^*(\mathbf R \times M)$ such that $M$ is Hermitian positive semi-definite  $\nu$-a.e. and $\trace(M)=1$ $\nu$-a.e.,
\end{description}
such that, for every compactly supported pseudo-differential operator $B$ with principal symbol $b$ of order $2m$ we have
\begin{equation}
\label{mdm}
\lim_{k\rightarrow +\infty} \left\langle B u_{n_k}| u_{n_k} \right\rangle_{H^{-m},H^m}=\int_{S^*(\mathbf R \times M)}\trace(bM)\dd\nu\text{.}
\end{equation}

\end{proposition}

Note that here $b$ is a matrix of dimension $n$ depending on $(t,\tau,x,\xi)$. One crucial property is that $(u_{n_k})_k$ strongly converges to $0$ if and only if $\mu=0$. 

\begin{definition}
In the setting of the previous theorem we will call $\mu=M\nu$ the microlocal defect measure of the sub-sequence $(u_{n_k})_k$ and we will say that $(u_n)_n$ is ``pure'' if it has a microlocal defect measure without preliminary extraction of a sub-sequence.
\end{definition}

\begin{proposition}\label{supportP}
Let $I\subset \mathbf R$ be a compact interval and $(u_n)$ be a pure sequence of $H^1(I\times M )$ weakly converging to $0$ with $M\dd\nu$ as microlocal defect measure. Recall that $P=\dt^2-\Delta$ and that its principal symbol is $p\cdot\mathrm{Id}$, the following properties are equivalent :
\begin{description}
\item[(i)] $Pu_n \underset{n \rightarrow \infty}{\rightarrow} 0$ strongly in $H^{-1}(I\times M)$.
\item[(ii)] $\nu$ is supported on the set $\{p=0\}$.
\end{description} 
\end{proposition}

\begin{proposition}
Let $(u_k)_k$ be a bounded sequence of $H^1( I \times M)$ weakly converging to $0$. Assume that $u_k$ is solution of the damped wave equation for every $k$ and let $b$ be a smooth function on $S^*( I \times M)$ to $\mathscr M_n(\mathbf C)$, $1$-homogeneous in the $(\tau,\xi)$ variable. If $(u_k)$ is pure with microlocal defect measure $\mu=M\nu$ then

\[
\int_{S^*( I \times M)} \trace \Big[(\{b,p\}-2\tau(ab+ba))M \Big] \dd \nu =0 \pp
\]
\end{proposition}

\begin{proof}
Let $B$ be a pseudo-differential operator of order $1$ and with principal symbol $b$, we then have 
\[
\lim_{k\to \infty} \left\langle[B,P]u_k,u_k\right\rangle_{H^{-1},H^1}=\int \trace[\sigma_2([B,P])M]\dd \nu=\frac{1}{i}\int\trace\left[\{b,p\}M\right]\dd \nu\vv 
\]
but we moreover know that $\langle [B,P]u_k,u_k\rangle=-2\langle(Ba\partial_t+a\partial_tB)u_k,u_k\rangle$, which tends to 
\[
-2i\int  \trace[\tau(ab+ba)M]\dd \nu\vv
\] 
thus finishing the proof.
\end{proof}
In what follows $\mu=M\dd \nu$ will denote the microlocal defect measure of a pure sequence $(u_k)_k$ of solutions of the damped wave equation on $\mathbf R\times M$. Here our aim is to give a relation between $\varphi_s^*\mu$ and $\mu$. The measure $\varphi_s^*\mu$ is the push forward of $\mu$ by $\varphi_s$, it is defined by the following property
\[
\text{for every }\mu\text{-integrable function } b \text{ we have } \;\int \trace[(b\circ \varphi_s) \dd\mu]=\int \trace[b\dd\varphi_s^*\mu]\pp
\]
\begin{definition}
For every $s\in \bb R$ we define the function $G_s:T^*(\mathbf R \times M) \to \mathscr{M}_n(\bb C) $ as the solution of the following differential equation. 
\[
\left\lbrace
\begin{array}{l}
G_0(t,\tau,x,\xi)=\mathrm{Id}_n \\
\partial_s G_s(t,\tau,x,\xi)=\{p,G_s\}(t,\tau,x,\xi)+2\tau G_s(t,\tau,x,\xi)a(x)\pp
\end{array}
\right.
\]

\end{definition}
The matrix $G_t$ is a cocycle map, that is, it satisfies the relation $G_{s+t}(\rho)=G_t(\varphi_s(\rho))G_s(\rho)$. The proof of this fact is given for $G^+_t$ at the end of the section.
\begin{proposition}
The propagation of the measure is given by the formula $\varphi_s^*\mu=G_{-s}\mu G_{-s}^*$, more precisely this means that for every continuous function $b$ compactly supported in the $(t,x)$ variable we have
\[
\int_{S^*(\mathbf R \times M)} \trace[(b\circ\varphi_s)G_sMG_s^*]\dd \nu =\int_{S^*(\mathbf R \times M)} \trace[bM]\dd \nu
\]
or equivalently for every continuous function $c$ compactly supported in the $(t,x)$ variable
\[
\int_{S^*(\mathbf R \times M)} \trace[c \,G_{-\sigma}MG_{-\sigma}^*]\dd \nu=\int_{S^*(\mathbf R \times M)} \trace[c\circ \varphi_\sigma M]\dd \nu\pp
\]
\end{proposition}

\begin{proof}
In order to show the first equality it suffice to verify that
\begin{equation}\label{deriveeintegrale}
\partial_s \int \trace[(b\circ\varphi_s)G_sMG_s^*]\dd \nu =0\pp
\end{equation}
We know that we can differentiate under the integral sign,
\[
\partial_s\int \trace\big[(b\circ\varphi_s)G_sMG_s^*\big]\dd \nu =\int \trace\Big[\partial_s\big((b\circ\varphi_s)G_sMG_s^*\big)\Big]\dd \nu=\int \trace\Big[\partial_s\big(G_s^*(b\circ\varphi_s)G_s\big)M\Big]\dd \nu \pp
\] 
Denoting by a $'$ the differentiation with respect to $s$ we then get 
\[
\begin{array}{rcl}
\displaystyle\partial_s \big(G_s^*(b\circ\varphi_s)G_s\big)&= &\displaystyle{G_s^*}'(b\circ\varphi_s)G_s + G_s^* \{p,b\circ \varphi_s\}G_s + G_s^*(b\circ\varphi_s)G_s'\\
\, &=&\displaystyle \{p,G_s^*(b\circ\varphi_s)G_s\} - \{p,G_s^*\}(b\circ\varphi_s)G_s- G^*_s(b\circ\varphi_s)\{p,G_s\}\\
\, &\,&\displaystyle +{G_s^*}' (b\circ \varphi_s)G_s + G_s^*(b\circ\varphi_s)G_s'\vv
\end{array}
\]
and by application of the previous proposition 
\[
\int \trace[\{p,G_s^*(b\circ\varphi_s)G_s\}M]\dd \nu = -\int\trace[(2\tau aG_s^*(b\circ\varphi_s)G_s+2\tau G_s^*(b\circ\varphi_s)G_s a)M  ]\dd\nu\pp
\]
By gathering all these terms we see that in order to have \eqref{deriveeintegrale} it suffices that
\[
\partial_s G_s=\{p,G_s\}+2\tau G_s a \;\text{ and }\; \partial_s{G_s^*}=\{p,{G_s^*}\}+2\tau a {G_s^*}\vv
\]
which coincides with the definition of $G$ and proves the first formula. The last formula is obtained by simply writing $c=b\circ\varphi_s$ and $\sigma=-s$.
\end{proof}

\begin{proposition}
The measure $\nu$ is supported on the set $\{\tau=\pm 1/2\}$. 
\end{proposition}

\begin{proof}
It a consequence of the proposition \ref{supportP} : $\nu$ is a measure on $S^*(\mathbf R \times M)$ so $\tau^2+|\xi|^2_g=1/2$ and it is supported on the set $\{p=0\}$ because $(\partial_t^2-\Delta) u_k=-2a\partial_t u_k$ strongly converges to $0$ in $H^{-1}$.
\end{proof}

\begin{definition}
This encourages us to consider the two connected components
\[
SZ^+=S^*(\mathbf R \times M)\cap\{\tau=-1/2\} \text{ and } SZ^-=S^*(\mathbf R \times M)\cap \{\tau=1/2\},
\]
as well as $\mu^+=M^+\nu^+$ and $\mu^-=M^-\nu^-$ the restrictions of $\mu$ to $SZ^+$ and $SZ^-$. Moreover we will respectively note $G_s^+$ and $G_s^-$ the restrictions of $G_s$ to $SZ^+$ and $SZ^-$.
\end{definition}

With this notation we get
\[
\partial_s G_s^+=\{p,G_s^+\}-G_s^+ a\pp
\]
\begin{remarque}
Since the function $a$ only depends on $x$ and since the $\tau$ variable is constant on $SZ^+$ and $SZ^-$, the functions $G_s^+$ and $G_s^-$ only depends on $(x,\xi)$ so we can also consider them as functions on $S^*M$.
\end{remarque}

\begin{corollaire}\label{propagationmesure}
Let $B$ be a Borel set of $SZ^+$ we have $\displaystyle \nu^+(\varphi_s(B))=\int_B \trace[G_{s}^+ M{G_{s}^+}^*]\dd \nu^+$.
\end{corollaire}

\begin{proof}
\[
\nu^+(\varphi_s(B))=\int_{SZ^+}\mathbf 1_{\varphi_s(B)}\dd \nu^+ = \int_{SZ^+}\mathbf{1}_B\circ \varphi_{-s}\dd \nu^+=\int_{SZ^+}\mathbf{1}_B\trace[G_s^+M{G_s^+}^*]\dd \nu^+
\]

\end{proof}

The cocycle $G^+$ thus plays an important role here since it completely describes the evolution of the microlocal defect measure. We finish this section with a few useful remarks about $G^+$.

\paragraph*{} A direct calculation shows that the matrix $G^+$ satisfy the following cocycle formula : 
\begin{equation}\label{eq:cocycleformula}
\forall \rho\in S^*M, \; \forall s,t\in \mathbf R,\; \; G_{s+t}^+(\rho)=G^+_{t}(\phi_s(\rho))G^+_{s}(\rho)\pp
\end{equation}
Indeed if we differentiate the right side with respect to $s$ we get 
\[
\begin{array}{rcl}
\partial_s G^+_{t}(\phi_s(\rho))G^+_{s}(\rho) &=&\displaystyle G_t^+(\phi_s(\rho))\big[\{p,G^+_s\}(\rho)-G_s^+(\rho)a(\rho)\big]+\{p,G_t^+\circ\phi_s\}(\rho)G_s^+(\rho)\\
\, &=&\displaystyle \{p, (G^+_{t}\circ\phi_s)G^+_{s}\}(\rho)-G^+_{t}(\phi_s(\rho))G^+_{s}(\rho)a(\rho)\pp
\end{array}
\]
The matrices $(G^+_{t}\circ\phi_s)G^+_{s}$ and $G_{s+t}^+$ thus satisfy the same differential equation with the same initial condition and are consequently  equal. This cocycle formula gives us a second differential equation satisfied by $G^+$. For every $(x_0,\xi_0)\in S^*M$ 
\[
\partial_t G_t^+(x_0,\xi_0)=\lim_{h\to 0} \frac{G_{t+h}^+(x_0,\xi_0)-G_t^+(x_0,\xi_0)}{h}\;\; \text{ and } \;\; G_{t+h}^+(x_0,\xi_0)=G_h^+(\phi_t(x_0,\xi_0))G_t^+(x_0,\xi_0)
\]
\[\text{ hence } \; \partial_t G_t^+(x_0,\xi_0)=\left.\partial_s G_s^+(\phi_t(x_0,\xi_0))\right|_{s=0} \cdot G_t^+(x_0,\xi_0) =-a(x_t)G_t^+(x_0,\xi_0)\vv \]
where $(x_t,\xi_t)=\phi_t(x_0,\xi_0)$. In accordance with the definition of $G^+$ given in the introduction we see that it is the solution of the differential equation
\begin{equation}\label{equationG}
\left\lbrace
\begin{array}{l}
G^+_0(x_0,\xi_0)=\mathrm{Id}_n \\
\partial_t G^+_t(x_0,\xi_0)=-a(x_t)G^+_t(x_0,\xi_0)\pp
\end{array}
\right.
\end{equation}
Let us add a last formula which will be useful for later. If we define $j:(x,\xi)\mapsto (x,-\xi)$ we have $\phi_s(j(\rho))=j(\phi_{-s}(\rho))$ and we deduce that $\partial_s (G_s^- \circ j) = -\{p,G_s^-\circ j\} + (G_s^-\circ j)a$.

\section{Estimation of the best decay rate}\label{section3}

\paragraph*{} Recall some definitions of the introduction. The following quantities are non-positive :
\begin{equation}
D(R)=\sup \{ \mathfrak{Re}(\lambda_j) : \lambda_j \in \mathrm{sp}(A_a), |\lambda_j|> R \}\vv \;\; D_0=\lim_{R\to 0^+} D(R) \; \text{ and } \; D_\infty=\lim_{R\to\infty}D(R)\pp
\end{equation}
For every $t\geq 0$ we chose $y_t$ a vector of $\mathbf C^n$ of euclidean norm $1$ such that 
\begin{equation}
G^+_{t}(x_0,\xi_0)G^+_{t}(x_0,\xi_0)^* y_t= \|G^+_{t}(x_0,\xi_0)\|^2_2y_t \pp
\end{equation}
The vector $y_t$ depends on $(x_0,\xi_0)$, even though it is not written. We then define for every $t>0$ the quantities 
\begin{equation}
C(t)=\frac{1}{t} \inf_{(x_0,\xi_0)\in S^*M} \int_0^t \langle a(x_s)y_s,y_s\rangle\dd s=\frac{-1}{t}\sup_{(x_0,\xi_0)\in S^*M} \ln \left( \|G_t^+(x_0;\xi_0)\|_2 \right)\, \text{ and }\, C_\infty=\lim_{t\to \infty} C(t) \pp
\end{equation}
We will see later that these definitions make sense and that they do not depend on the choice of $y_s$. Remember that $C(t)$ is non-negative. 

The remainder of this section is mainly dedicated to the proof of the formula for $\alpha$. Before starting let us just indicate the main steps of the proof. We first give an upper bound of $\alpha$ using Gaussian beams (also called coherent states). These are particular approximate solutions of the damped wave equation that are concentrated near a geodesic. In order to proves the lower bound of $\alpha$ we will use a high frequency inequality (Proposition \ref{lemmeinegaliteahauteenergie}) together with a decomposition  of solutions of \eqref{dampedwaveequation} in high and low frequencies. 

\subsection{Upper bound for $\alpha$}

\paragraph*{}Let $\lambda_j\in \mathrm{sp}(A_a)\backslash \{0\}$ and $u=(u_0,u_1)\in E_{\lambda_j}\backslash\{0\}$ be such that $A_a u=\lambda_j u$. The solution of \eqref{dampedwaveequation} then is $u(t,x)=e^{t\lambda_j}u_0(x)$ and we have $E(u,t)=e^{2t\mathfrak{Re}(\lambda_j)}E(u,0)$. Since $E(u,0)\neq 0$ we know that $\alpha\leq -2D(0)$. 

\paragraph*{} Showing that $\alpha\leq 2C_\infty$ is a bit more difficult as it requires us to construct Gaussian beams. We will start by constructing them on $\mathbf R^d$ endowed with a Riemaniann metric $g$. Gaussian beams are approximate solutions of the wave equation (in a sens made precise by \eqref{eq:solapprochee}) whose energy may be arbitrarily concentrated along a geodesic up to a fixed time $T>0$ (see \eqref{eq:concentrationenergiesolutionapproche}). They will allow us to construct exact solutions to the damped wave equation whose energy is also arbitrarily concentrated along a geodesic up to some time $T$. As always we will call $(x_t;\xi_t)=\phi_t(x_0,\xi_0)$ the points of the geodesic. We will follow and adapt the construction given in \cite{ral82} or \cite{mazu02} to fulfill our needs. 

\paragraph*{}We consider for every integer $k$ a function $u_k :  \mathbf R^d \to \mathbf R^n$ given by the formula 
\[
u_k(t,x)=k^{-1+d/4}b(t,x)\exp(ik \psi(t,x)) \omega
\]
where $\psi(t,x)=\langle\xi(t),(x-x(t)\rangle+\frac{1}{2}\langle M_t(x-x(t)),x-x(t)\rangle$ with $M_t$ a $d\times d$ symmetric matrix with positive definite imaginary part, $b$ is a continuous bounded function and $\omega$ is a vector of $\mathbf C^n$. In what follows $C$ represents a positive constant that can vary from one line to another but does not depends on $k$, however $C$ can depend on $T$.

\begin{theo}[\cite{ral82}]
It is possible to chose $M_t$ and $b$ such that 
\begin{equation}\label{eq:solapprochee}
\sup_{t\in [0;T]}\|\partial_t^2 u_k(t,\cdot)-\Delta_gu_k(t,\cdot)\|_{L^2(\mathbf R^d)}\leq Ck^{-1/2}\vv
\end{equation}
\begin{equation}\label{energiesolapprochee}
 \forall t\in [0;T] \; \lim_{k\to \infty}E(u_k,t) \;\text{ is positive, finite and does not depends on }t\vv
\end{equation}
\begin{equation}\label{eq:concentrationenergiesolutionapproche}
\sup_{t\in [0;T]} \int_{\mathbf R^d\backslash B(x_t,k^{-1/4})} |\dt u_k(t,\cdot)|^2+|\nabla u_k(t,\cdot)|_g^2 \dd x \leq C\exp(-\beta \sqrt k)\pp
\end{equation}
\end{theo}
Under these conditions we say that $u_k$ is a Gaussian beam. We also need a lemma of \cite{ral82}.
\begin{lemme}[\cite{ral82}]\label{lem:lemme 2}
Let $c\in L^\infty(\mathbf R^d)$ be a function satisfying $|x-x_0|^{-\alpha}c(x) \in L^\infty(\mathbf R^d)$ for some $\alpha \geq 0$ and some $x_0\in \mathbf R^d$, and let $A$ be a symmetric, positive definite, real $d\times d $ matrix. Then
\begin{equation}
\int_{\mathbf R^d}\left|c(x)\exp\big(-k\langle M(x-x_0),x-x_0\rangle\big)\right|^2\dd x \leq C k^{-d/2-\alpha}
\end{equation}
for some $C>0$ that does not depend on $k$.
\end{lemme}

Using lemma \ref{lem:lemme 2} with $c=|b(t,\cdot)|$ and $\alpha=0$ we see that $\|u_k(t,\cdot)\|_{L^2(\mathbf R^d)}\leq C k^{-1/2}$. Let us now define the function  $v_k(t,x)=G_t^+(x_0,\xi_0)u_k(t,x)$, as we shall see it is an approximate solution of the damped wave equation. Indeed we have 
\[
(\dt^2-\Delta_g+2a\dt)v_k(t,x)=G_t^+(x_0,\xi_0)\left(\dt^2-
\Delta_g \right)u_k(t,x)+2(a(x)-a(x_t))G_t^+(x_0,\xi_0)\dt u_k(t,x)
\]
\[
 +\left(a(x_t)^2-\dt a(x_t)-2a(x)a(x_t)\right)G_t^+(x_0,\xi_0)u_k(t,x)\overset{\mathrm{def}}{=}f_k(t,x)
\]
and we need to show that $\|f_k(t,\cdot)\|_{L^2}\leq C k^{-1/2}$. In order to do that we only need to prove $ \|2(a(\cdot)-a(x_t))G_t^+(x_0,\xi_0)\dt u_k(\cdot,t)\|_{L^2}\leq Ck^{-1/2}$ because the other terms obviously satisfy the bound. Now since the function $x\mapsto|x-x_t|^{-1}\|a(x)-a(x_t)\|_2$ is in $L^\infty$ we can use lemma \ref{lem:lemme 2} on $2(a(\cdot)-a(x_t))G_t^+(x_0,\xi_0)\dt u_k(\cdot,t)$ and we finally get
\[
\sup_{t\in[0;T]}\|(\dt^2-\Delta_g+2a\dt)v_k(t,\cdot)\|_{L^2(\mathbf R^d)}\leq C k^{-1/2}\pp
\]
Moreover we see that $v_k$ still satisfies the properties \eqref{energiesolapprochee} and \eqref{eq:concentrationenergiesolutionapproche}, although now the limit of the energy of $v_k$ may vary with $t$ because $G_t^+(x_0,\xi_0)$ does. We finally define $w_k$ as the solution of \eqref{dampedwaveequation} with initial conditions $w_k(0,\cdot)=v_k(0,\cdot)$ and $\dt w_k(0,\cdot)=\dt v_k(0,\cdot)$. By definition of $w_k$ we have $(\dt^2-\Delta_g+2a\dt)v_k=(\dt^2-\Delta_g+2a\dt)(v_k-w_k)=f_k$ and thus
\[
\frac{\dd}{\dd t}E(v_k-w_k,t)=-2\int_{\mathbf R^d}\langle a\dt (v_k-w_k),\dt (v_k-w_k)\rangle\dd x +\int_{\mathbf{R}^d}\mathfrak{Re}\langle f_k,\dt (v_k-w_k)\rangle\dd x \pp
\]
The first term of the right hand side is negative and, using Cauchy-Schwarz, we can bound the second term by $Ck^{-1/2}$. Indeed we already know that $\|f_k\|_{L^2}\leq Ck^{-1/2}$ and $\|\partial_t(v_k-w_k)\|_{L^2}$ is uniformly bounded in $k\in \mathbf N$ and $t\in [0;T]$. Since $E(w_k-v_k,0)=0$ by integrating we get
\[ \sup_{t\in [0;T]} E(v_k-w_k,t)\leq C Tk^{-1/2}\pp
\]
In combination with the estimate \eqref{eq:concentrationenergiesolutionapproche} of $u_k$ we see that $w_k(t,\cdot)$ is concentrated around $x_t$, more precisely we have
\begin{equation}\label{gaussianbeamenergylocalisation}
\sup_{t\in [0;T]} \int_{\mathbf R^d\backslash B(x_t,k^{-1/4})} |\dt w_k(t,\cdot)|^2+|\nabla w_k(t,\cdot)|_g^2 \dd x \leq C T k^{-1/2}\pp
\end{equation}
Then we set $\omega$ such that $\lim_{k\to\infty}E(v_k,0)=1$ and $G^+_T(x_0,\xi_0)\omega=\|G^+_T(x_0,\xi_0)\|_2\omega$. According to the definition of $v_k$ we have
\[
E(v_k,T)=\frac{1}{2}\int_M |G_T^+(x_0,\xi_0)\dt u_k(T,\cdot)-a(x_T)G_T^+(x_0,\xi_0)u_k(T,\cdot)|^2+|G_T^+(x_0,\xi_0)\nabla u_k(T,\cdot)|^2 \dd x
\]
but $\|u_k(T,\cdot)\|_{L^2}\leq C k^{-1/2}$ so the term $a(x_T)G_T^+(x_0,\xi_0)u_k(T,\cdot)$ vanishes and we get \[\lim_{k\to \infty }E(v_k,T)= \|G^+_T(x_0,\xi_0)\|_2^2 \pp\] This in turn imply that $(w_k)_k$ is sequence of solutions to \eqref{dampedwaveequation} which satisfies $\lim_{k\to \infty}E(w_k,0)=1$ and $\lim_{k\to \infty }E(w_k,T)= \|G^+_T(x_0,\xi_0)\|_2^2$. Summing up the discussion so far, we have

\begin{proposition}
For any time $T>0$, any $\varepsilon >0$ and any $(x_0,\xi_0)\in S^*\mathbf R^d$ there exists a solution $u$ of the damped wave equation such that $E(u,0)=1$ and $\left|E(u,T)-\|G^+_T(x_0,\xi_0)\|_2^2\right|<\varepsilon$.
\end{proposition}

Using charts this result extends to the case of a compact Remannian manifold $(M,g)$ and we finally get 

\begin{proposition}\label{gaussianbeamprop}
For any time $T>0$, any $\varepsilon >0$ and any $(x_0,\xi_0)\in S^*M$ there exists a solution $u$ of the damped wave equation such that $E(u,0)=1$ and $\left|E(u,T)-\|G^+_T(x_0,\xi_0)\|_2^2\right|<\varepsilon$.
\end{proposition}

Define $\Gamma_t=G^+_t(x_0,\xi_0)G^+_t(x_0,\xi_0)^*$ and, for every time $t$, chose $y_t$ a vector of euclidean norm $1$ such that $\displaystyle \Gamma_ty_t=\|\Gamma_t\|_2y_t$. Let us stress again that $y_t$ and $\Gamma_t$ both implicitly depends on $(x_0,\xi_0)$.  

\begin{proposition}\label{prop:formulenormeG}
\[\|G_t^+(x_0,\xi_0)\|_2^2=\|\Gamma_t\|_2=\exp\left(-2\int_0^t \big\langle a(x_s)y_s,y_s\big\rangle\dd t\right)\]
\end{proposition}

\begin{proof}
The only thing to prove is the second equality. The map  $t\mapsto \Gamma_t$ is the solution of the differential equation
\begin{equation}\label{defmatB}
\left\lbrace
\begin{array}{l}
\Gamma_0=\mathrm{Id}_n \\
\dt \Gamma_t=-a(x_t)\Gamma_t-\Gamma_ta(x_t)\vv
\end{array}
\right.
\end{equation}
it is hence $C^\infty$ and \textit{a fortiori} locally Lipschitz. Consequently the map $t\mapsto \|\Gamma_t\|_2$ is also locally Lipschitz\footnote{We cannot  really do better than that in terms of regularity.}, this imply that it is differentiable for almost every $t$. Since $\Gamma_t$ is hermitian positive definite $\|\Gamma_t\|_2=\langle\Gamma_ty_t,y_t\rangle$ and if $z$ is any other vector of norm $1$ then $\|\Gamma_t\|_2\geq \langle\Gamma_tz,z\rangle$. Fix a time $t_0$, we then have
\[\begin{array}{rcl}
\left.\dt \langle\Gamma_ty_{t_0},y_{t_0}\rangle\right|_{t=t_0} &=& -\big\langle[a(x_{t_0})\Gamma_{t_0}+\Gamma_{t_0}a(x_{t_0})] y_{t_0},y_{t_0}\big\rangle\\
\, &=& -2\|\Gamma_{t_0}\|_2\langle a(x_{t_0})y_{t_0},y_{t_0}\rangle\pp
\end{array}
\]
We know that $\langle\Gamma_ty_t,y_t\rangle\geq \langle\Gamma_t y_{t_0},y_{t_0}\rangle$ for every $t$ and there is equality when $t=t_0$. If $\|\Gamma_t\|_2$ is differentiable at $t_0$ we deduce that at this point the derivatives of the two functions $t\mapsto \langle\Gamma_ty_t,y_t\rangle$ and $t\mapsto\langle\Gamma_t y_{t_0},y_{t_0}\rangle$ must be the same. Hence for almost every time $t$ 
\[
\dt \|\Gamma_t\|_2= \dt \langle\Gamma_t y_t,y_t\rangle = -2\|\Gamma_t\|_2\langle a(x_t)y_t,y_t\rangle \pp
\]
To finish the proof we just need to see that the function
\[
\Phi:t\mapsto \frac{\|\Gamma_t\|_2}{\displaystyle\exp\left(-2\int_0^t \big\langle a(x_s)y_s,y_s\big\rangle\dd s\right)}
\]
is Lipschitz on every bounded interval $[0;T]$ and \textit{a fortiori} absolutely continuous. From $\Phi'=0$ a.e. we deduce that $\Phi$ is constant and since $\Phi(0)=1$ this finishes the proof.

\end{proof}
Notice that the choice of $y_t$ is not unique and that $t\mapsto y_t$ is not continuous in general. On the other hand the derivative of $\|\Gamma_{t}\|_2$ is uniquely defined almost everywhere, so that the choice of $y_t$ has no importance. 
Therefore we have
\[
C(t)\overset{\mathrm{def}}{=}\frac{-1}{t}\sup_{(x_0,\xi_0)\in S^*M} \ln \left( \|G_t^+(x_0;\xi_0)\|_2 \right)=\frac{1}{t} \inf_{(x_0,\xi_0)\in S^*M} \int_0^t \langle a(x_s)y_s,y_s\rangle\dd s\pp
\]
This function is obviously non-negative but in order to proves other properties it is easier to work with $\exp(-tC(t))=\sup_{\rho\in S^*M}\|G_t^+(\rho)\|_2$ . The function $a$ is continuous on $M$ and the geodesic flow $\phi$ is continuous on $\mathbf R\times S^*M$, since $G^+$ is defined as the solution of \eqref{equationG} the function $\|G^+\|$ is in turn continuous on $\mathbf R\times S^*M$. As $S^*M$ is compact, $t\mapsto \exp(-tC(t))$ is continuous and so is $t\mapsto C(t)$. We now show that $t\mapsto tC(t)$ is sub-additive : let $t$ and $s$ be two non negative reals, we have the following equivalences :
\[
(t+s)C(t+s)\geq tC(t)+sC(s) \Longleftrightarrow \exp(-2(t+s)C(t+s))\leq \exp(-2tC(t))\exp(-2sC(s))
\]
\begin{equation}\label{sousadditiviteinegalite}
\Longleftrightarrow \sup_{(x,\xi)\in S^*M} \|G_{t+s}^+{G_{t+s}^+}^*\|_2 \leq \big(\sup_{(x,\xi)\in S^*M} \|G_{t}^+{G_{t}^+}^*\|_2\big)\cdot \big(\sup_{(x,\xi)\in S^*M} \|G_{s}^+{G_{s}^+}^*\|_2\big)\pp
\end{equation}
Recall the cocycle formula $G_{s+t}^+(\rho)=G^+_{t}(\phi_s(\rho))G^+_{s}(\rho)$, it follows that
\[
G_{t+s}^+(\rho){G_{t+s}^+(\rho)}^*=G_t^+(\phi_s(\rho))G_s^+(\rho){G_s^+(\rho)}^*{G_t^+(\phi_s(\rho))}^*
\]
and since for any two matrices $R$ and $S$ we have $\|S^*R^*RS\|_2\leq \|S^*S\|_2\|R^*R\|_2$, the inequality \eqref{sousadditiviteinegalite} is satisfied and $t\mapsto tC(t)$ is indeed sub-additive. By application of Fekete's sub-additive lemma we deduce that $C(t)$ admits a limit when $t\to \infty$ and that $C(t)\leq C_\infty$ for every positive $t$.

By combining the results of this section it is now easy to prove that $\alpha\leq 2C_\infty$. Assume that $\alpha=2C_\infty+4\eta$ for some $\eta >0$, this means that there exists some constant $C>0$ such that 
\begin{equation}\label{jesaispasquoimettrecommelabel}
\forall t\geq0,\; \forall u \in H,\;E(u,t)\leq CE(u,0)\exp(-2t(C_\infty+\eta)). 
\end{equation}
Now pick some $T$ such that $C\exp(-2T(C_\infty+\eta))<\exp(-T(2C_\infty+\eta))$. Since $C_\infty\geq C(T)$ we have $\exp(-T(2C_\infty+\eta))\leq\exp(-T(2C(T)+\eta))$ but using proposition \ref{gaussianbeamprop} there exist some $u\in H$ such that 
\[
E(u,T)>E(u,0)\exp(-T(2C(T)+\eta))>CE(u,0)\exp(-2T(C_\infty+\eta))
\]
which contradicts \eqref{jesaispasquoimettrecommelabel} and concludes the proof of $\alpha\leq 2 C_\infty$.

\subsection{Lower bound for $\alpha$}

We are now going to use the results of Section $2$ in order to prove the following energy inequality for the high frequencies.  

\begin{proposition}\label{lemmeinegaliteahauteenergie}
For every time $T>0$ and every $\varepsilon >0$ there exists a constant $C(\varepsilon,T)$ such that for every $u=(u_0;u_1)$ in $H$ we have

\begin{equation}\label{inegaliteahauteenergie}
E(u,T)\leq (1+\varepsilon)\displaystyle e^{-2TC(T)} E(u,0)+C(\varepsilon,T)\|u_0,u_1\|^2_{L^2\oplus H^{-1}}\pp
\end{equation}
\end{proposition}

\begin{proof}
Assume that \eqref{inegaliteahauteenergie} is false, in this case for some $T$, some $\varepsilon$ and every integer $k\geq 1$ there is a solution $u^k=(u_0^k,u_1^k)$ of \eqref{dampedwaveequation} satisfying 
\begin{equation}\label{3.3}
E(u^k,T)\geq (1+\varepsilon)\displaystyle e^{-2TC(T)}E(u,0)+ k\|u_0^k,u_1^k\|^2_{L^2\oplus H^{-1}} \; \text{ and }\; E(u^k,0)=1\pp
\end{equation}
First we show that the sequence $(u^k)$ is bounded in $H^1(I\times M)$, where $I=[-2T;2T]$. Indeed $E(u^k,0)=1$ and \eqref{formuleenergie} implies that $E(u^k,-2T)$ is bounded uniformly in $k$. Since the energy is non increasing the sequence $(u^k)$ must be bounded in $H^1(I\times M)$. Moreover $\|u_0^k,u_1^k\|^2_{L^2\oplus H^{-1}}\leq E(u^k,T)/k\leq 1/k$, so $(u^k)$ converges to $0$ in $L^2(I\times M)$ and so it weakly converges to $0$ in $H^1(I\times M)$. If we are to extract a sub-sequence we might as well assume that $(u^k)$ admits $\mu=M\nu$ (with $\trace(M)=1$) as microlocal defect measure. As the energy is non increasing it follows from \eqref{3.3} that for every $\eta\in]0;T[$ and every non negative function $\psi\in C^\infty_0(]0;\eta[)$,
\[
\int_{T-\eta}^T \psi(T-t)E(u^k,t)\dd t \geq (1+\varepsilon)e^{-2TC(T)}\int_0^\eta \psi(t)E(u^k,t) \dd t\pp
\]
Since this is true for every function $\psi$, taking the limit $k\to\infty$ in the previous inequality gives
\begin{equation}\label{3.5}
\nu(S^*(\mathbf R \times M)\cap t\in]T-\eta,T[)\geq (1+\varepsilon)e^{-2TC(T)}\nu(S^*(\mathbf R \times M)\cap t\in]0;\eta[)\pp
\end{equation}
On the other hand Corollary \ref{propagationmesure} gives us
\[
\begin{array}{rcl}
\nu^+(SZ^+\cap t\in ]T-\eta;T[)&=& \nu^+(\varphi_{T-\eta}(SZ^+\cap t\in ]0;\eta[))\\
\, &=& \displaystyle \int_{SZ^+\cap t\in ]0;\eta[} \trace[G_{T-\eta}^+M{G_{T-\eta}^+}^*]\dd\nu \\ 
\; &  \leq &\displaystyle \sup_{(x;\xi)\in S^*(M)} \|G_{T-\eta}^+(x,\xi)\|_{2}^2 \nu^+(SZ^+\cap t\in ]0;\eta[)\\
\; & = & e^{-2(T-\eta)C(T-\eta)}\nu^+(SZ^+\cap t\in ]0;\eta[)\pp

\end{array}
\]
To get this upper bound we used the following properties.
\[
\trace[G_{T-\eta}^+M{G_{T-\eta}^+}^*]=\trace[{G_{T-\eta}^+}^*G_{T-\eta}^+M]\leq \|G_{T-\eta}^+\|_2^2\trace(M) =\|G_{T-\eta}^+\|_2^2
\]
We then use the same argument on $SZ^-$. With the relation  $\partial_s (G_s^- \circ j) = -\{p,G_s^-\circ j\} +2 (G_s^-\circ j)a$ given at the end of section 2 we find
\[
\begin{array}{rcl}
\nu^-(SZ^-\cap t\in ]T-\eta;T[)\leq e^{-2(T-\eta)C(T-\eta)} \nu^-(SZ^-\cap t\in ]0;\eta[)\pp

\end{array}
\]
By combining $\nu^+$ and $\nu^-$ together, we get
\begin{equation}\label{3.6}
\nu(S^*(\mathbf R \times M)\cap t\in]T-\eta,T[)\leq e^{-2(T-\eta)C(T-\eta)}\nu(S^*(\mathbf R \times M)\cap t\in]0;\eta[)\pp
\end{equation}
Recall that $\displaystyle s\mapsto e^{-2sC(s)}$ is continuous, so for $\eta$ sufficiently small the inequalities \eqref{3.5} and \eqref{3.6} imply that $\nu(S^*(\mathbf R \times M)\cap t\in]0;\eta[)=0$. Consequently the sequence $(u^k)$ strongly converges to $0$ in $H^1(]0;\eta[\times M)$ and thus it also strongly converges to $0$ in $H^1(I\times M)$. This contradicts the hypothesis $E(u^k,0)=1$ and finishes the proof.

\end{proof}

The remainder of the proof for the formula of $\alpha$ is completely borrowed from the article of Lebeau (\cite{leb93}), indeed it works verbatim\footnote{Although the article of Lebeau is in french, so any translation error that may occur is my mistake.}. Let $A_a^*$ be the adjoint of $A_a$, we have $-A_a^*=\begin{pmatrix}
0 & \mathrm{Id} \\
\Delta & +2a
\end{pmatrix}
$ and the spectrum of $A^*_a$ is the complex conjugate of the spectrum of  $A_a$. Let us call $E_{\lambda_j}^*$ the generalized eigenvector space of $A_a^*$ associated with the spectral value $\overline{\lambda_j}$. For $N\geq 1$ we set 
\[
H_N=\left\lbrace x\in H : (x|y)_H=0, \; \forall y \in \bigoplus_{|\lambda_j|\leq N}E_{\lambda_j}^*\right\rbrace \pp
\] 
The space $H_N$ is invariant by the evolution operator $e^{tA_a}$. To see that take $x\in H_N$ and $\{y_l\}$ a basis of the finite dimension vector space $\displaystyle \bigoplus_{|\lambda_j|\leq N} E_{\lambda_j}^*\subset D(A_a^*)$, we have 
\[
\dt (e^{tA}x|y_l)=(e^{tA}x|A_a^*y_l)=\sum c_{l,k}(e^{tA}x|y_k) \text{ and so } (e^{tA}x|y_l)=0 \pp
\]
Set $H'=L^2\oplus H^{-1}$ and let $\theta_n$ be the norm of the embedding of $H_N$ in $H'$. The operator $-A^*_a$ is a compact perturbation of the skew-adjoint operator $A_0$, this implies that the family $\{E^*_{\lambda_j}\}_j$ is total in $H$ (see \cite{gokr69}, chapter 5 theorem 10.1) and thus that $\lim \theta_N =0$. Let us assume that $2\min\{-D_0,C_\infty\}>0$, or otherwise there is nothing to prove. Fix $\eta>0$ small enough so that $\beta=2\min\{-D_0,C_\infty\}-\eta$ is positive. Now take $T$ such that $4|C_\infty-C(T)|<\eta$ and $2\log(3)<\eta T$ and finally $N$ such that $C(1,T)\theta^2_N\leq e^{-2TC(T)}$. It follows from the previous proposition that  
\[
\forall u\in H_N, \; E(u,T)\leq 3e^{-2TC(T)}E(u,0),
\]
and since $H_N$ is stable by the evolution 
\[
\forall k \in \mathbf N ,\;\forall u\in H_N,\; E(u,kT)\leq 3^ke^{-2kTC(T)}E(u,0)\pp
\]
The energy is non increasing, so there exists a real $B>0$ such that 
\begin{equation}\label{decroissanceenergiehautefrequence}
\forall t\geq 0,\; \forall u\in H_N,\; E(u,t)\leq Be^{-\beta t } E(u,0).
\end{equation}
Let $\gamma$ be a path circling around $\{\lambda_j : |\lambda_j|\leq N\}$ clockwise and $\Pi=\frac{1}{2i\pi}\int_\gamma \frac{\dd \lambda}{\lambda-A_a}$ be the spectral projector on $W_N=\bigoplus_{|\lambda_j|\leq N} E_{\lambda_j}$. In this case $\Pi^*$ is the spectral projector of $A_a^*$ on $\bigoplus_{|\lambda_j|\leq N} E_{\lambda_j}^*$ and so for every $u\in H$, one has
\begin{equation}\label{decompositionhautefrequencebassefrequences}
v=\Pi u \in W_N,\; w=(1-\Pi)u\in H_N\;\text{ and }\; u=v+w\pp
\end{equation}
Now $W_N$ is of finite dimension and since $\beta\leq -2D(0)$ there exists some $C$ such that  
\begin{equation}\label{decroissanceenergiebassefrequence}
\forall u\in W_N,\; \forall t\geq0 ,\; E(u,t)\leq Ce^{-\beta t} E(u,0)\pp
\end{equation}
Finally, since the decomposition \eqref{decompositionhautefrequencebassefrequences} is continuous, there exists some $C_0$ such that $E(v,0)+E(w,0)\leq C_0 E(u,0)$. Combining  \eqref{decroissanceenergiebassefrequence} and \eqref{decroissanceenergiehautefrequence} we get $\alpha\geq \beta$, thus finishing the proof of the formula for $\alpha$. 

\subsection{End of the proof of Theorem \ref{bigtheoreme} and proof of Theorem \ref{stabforte}}
\paragraph*{} We still need to prove properties \textbf{(i)}, \textbf{(ii)} and \textbf{(iii)} of Theorem \ref{bigtheoreme}. For \textbf{(ii)} there is nothing to do since it is already done in \cite{leb93} in the case $n=1$, which is sufficient. For \textbf{(i)} we can assume $C_\infty >0$ or otherwise there is nothing to prove. Notice that $E_{\lambda_j}\subset H_N$ as soon as $|\lambda_j |>N$, together with \eqref{decroissanceenergiehautefrequence} it means that, for every $\beta<2C_\infty$ and for $N$ large enough 
\[
|\lambda_j|>N \Rightarrow 2\mathfrak{Re}(\lambda_j)\leq -\beta\pp
\]
This implies $D_\infty \leq -C_\infty$ and proves \textbf{(i)}. Before we get to the last point of Theorem \ref{bigtheoreme} we are going to prove Theorem \ref{stabforte}.

\begin{proof}[Proof of Theorem \ref{stabforte}]
We start by proving \textbf{(ii)}$\Rightarrow$\textbf{(i)} by contraposition. Assume that \textbf{(i)} is not satisfied, if there is no weak stabilization then obviously \textbf{(ii)} is false. We can thus assume that there exists a point $(x_0,\xi_0)\in S^*M$ and a vector $y\in \mathbf C^n$ of euclidean norm $1$ such that $a(x_s)y=0$ for every time $s$. This means we have
\[
\partial_t {G^+_t}^*(x_0,\xi_0)y=-{G^+_t}^*(x_0,\xi_0)a(x_t)y=0
\]
\[\text{ hence }\; \|{G^+_t}(x_0,\xi_0){G^+_t}^*(x_0,\xi_0)\|_2=\|{G^+_t}^*(x_0,\xi_0)\|_2^2=\sup_{|v|=1}({G^+_t}^*(x_0,\xi_0)v,{G^+_t}^*(x_0,\xi_0)v)=1\pp\]
This implies that for every positive $t$ one has $C(t)=C_\infty=0$ and thus, by Theorem \ref{bigtheoreme}, it implies $\alpha=0$. This in turn shows that there is not strong stabilization and proves \textbf{(ii)}$\Rightarrow$\textbf{(i)}.

Reciprocally, assume that condition \textbf{(i)} is satisfied. Then by a compactness argument there exists $T>0$ such that  for all $(x_0,\xi_0)\in S^*M$ and all $y\in \mathbf C^n$ of euclidean norm $1$ 
\[
\int_0^T\langle a(x_t)y,y \rangle\dd t >0 \pp
\]
We begin by proving $C_\infty >0$, since $C_\infty\geq C(t)$ for every $t$ it suffices to show that $C(T)$ is positive. Let us assume that $C(T)=0$, then there exists $(x_0,\xi_0)\in S^*M$ and $y\in \mathbf C^n$ of norm $1$ such that $\Gamma_0(x_0;\xi_0)y=\Gamma_T(x_0,\xi_0)y=y$. We recall that $\Gamma_t=G_t^+G_t^+{}^*$ and that, according to proposition \ref{prop:formulenormeG}, $\|\Gamma_t\|_2$ is non increasing. As $\|\Gamma_0(x_0,\xi_0)\|_2=\|\Gamma_T(x_0,\xi_0)\|_2=1$ we know that $\|\Gamma_t(x_0,\xi_0)\|_2=1$ for every $t\in[0;T]$. Using Gaussian beams in section 3.1 we have proved that, for every $\varepsilon >0$ there exists a solution $u$ of the damped wave equation such that $|E(u,t)-\|\Gamma_t(x_0;\xi_0)y\|_2|<\varepsilon$ for every $t\in[0;T]$. Since the energy is non increasing it means that, for every $t\in[0;T]$ we have $\|\Gamma_t(x_0;\xi_0)y\|_2=1$ and thus that $\Gamma_t(x_0;\xi_0)y=\|\Gamma_t(x_0;\xi_0)\|_2y=y$. 
In view of proposition \ref{prop:formulenormeG} this means that 
\[
0=C(T)=\|\Gamma_T(x_0,\xi_0)\|_2=\frac{1}{T}\int_0^T \langle a(x_t)y,y\rangle \dd t\vv
\]
which is absurd, so we must have $C_\infty >0$. The weak stabilization assumption implies that $A_a$ has no eigenvalue (except $0$) on the line $\{\mathfrak{Re}(z)=0\}$. It follows that the only possibility for $D_0$ to be zero is that $D_\infty $ is also zero. However we showed that $C_\infty >0$ and $C_\infty \leq -D_\infty$ so we have $-D_0>0$ and, by Theorem \ref{bigtheoreme}, we have strong stabilization.
\end{proof}

With this proof we see why $C_\infty >0 $ is equivalent to \eqref{GCC}, the geometric control condition. In dimension $n=1$ the geometric control condition is equivalent to strong stabilization (\cite{blr92}) which is in turn equivalent to $\alpha>0$. This means that the situation \textbf{(iii)} of Theorem \ref{bigtheoreme} cannot happen when $n=1$. To show that the situation $C_\infty >0$ and $D(0)=0$ does happen we will work on the circle $M=\mathbf R/ 2\pi \mathbf Z$. Let $k>0$ be a fixed integer and set $u_1(t,x)=e^{ikt}\sin(k x)$ and $u_2(t,x)=e^{ikt}\sin(kx+1)$. The function $u$ defined by 
\[
u(t,x)=\begin{pmatrix}
u_1(t,x)\\ u_2(t,x)
\end{pmatrix} \;\text{ is a solution of } \partial_t^2 u - \Delta u =0\pp 
\]
We now define $a(x)$ as the orthogonal projector on $u(0,x)^\perp$, this way we get
\[\forall (t,x)\in \mathbf R \times M,\;\;a(x)\partial_t u(t,x)=ike^{ikt}a(x)u(0,x)=0 \pp\] The function $u$ is thus a solution of the damped wave equation and we see that $ik$ is an eigenvalue of $A_a$. By construction $D_0=0$, however $\mathrm{ker}(a(x))$ is of dimension $1$ and not constant so the geometric control condition is satisfied. This forces $C_\infty$ to be positive and finishes the proof of Theorem \ref{bigtheoreme}. 

\begin{remarque}
Let us emphasize once again that, in the scalar case ($n=1$), the geometric control condition implies $a> 0$ on an open set and thus it also implies weak stabilization. On the other hand, when $n>1$ we can have the geometric control condition and no weak stabilization. This means that when $n=1$ Theorem \ref{stabforte} can be stated without the weak stabilization condition but it is necessary whenever $n>1$.
\end{remarque}

\section{Behavior of $C_\infty$}\label{section4}
\paragraph*{}In this section we are interested in the behavior of $C_\infty $ as a function of the damping term $a$. For this reason we will denote by $C_\infty(a)$ the constant $C_\infty$ associated with the damping term $a$ when needed. In the scalar case, things are pretty simple. If $a$ and $b$ are two damping terms and $\lambda\geq 0$ a real number we have $C_\infty(\lambda a)=\lambda C_\infty(a)$ and $C_\infty(a+b)\geq C_\infty(a)+C_\infty(b)$, this is a direct consequence of \eqref{eq:formuleGscalaire}. Moreover if $a$ and $b$ are such that $a\geq b$ pointwise then $C_\infty(a)\geq C_\infty(b)$. The vector case is more complicated since there is no simple expression for the matrix $G^+_t$. We will thus restrain our self to the study of a one dimensional example. 

\paragraph*{}We will work on the circle $M=\mathbf R/ 2\pi \mathbf Z$.  Using the cocycle formula of $G^+$ it's easy to see that $\lim \frac{-1}{t}\ln (\|G_{t}^+(x,\pm1)\|_2^2)$ does not depends on $x$, which will be taken equal to $0$ from now on. Still using this cocycle formula we see that if $p$ and $q$ are integers then 
\[
C_\infty(a)=\lim_{t\to \infty} \frac{-1}{t}\ln (\|G_{t}^+(0,\pm1)\|_2^2)=\lim_{p\to \infty} \frac{-1}{2p\pi}\ln (\|G_{2p\pi}^+(0,\pm1)\|_2^2)
\] 
and also \[G_{2(p+q)\pi}^+(0,\pm1)=G_{2p\pi}^+((0,\pm1))G_{2q\pi}^+((0,\pm1))\pp\]
Combining all that, we finally find
\[
\begin{array}{rcl}
\displaystyle\lim_{t\to \infty} \frac{-1}{t}\ln \left(\left\|G_{t}^+(x,\pm1)\right\|_2^2\right)&=&\displaystyle\lim_{p\to \infty} \frac{-1}{2p\pi}\ln \left(\left\|\left[G_{2\pi}^+(0,\pm1)\right]^p\right\|_2^2\right)\\
\,&=& \displaystyle \frac{-1}{\pi} \ln\left(\rho\left(G^+_{2\pi}(0,\pm1)\right)\right)
\end{array}
\]
where $\rho(M)$ denotes the spectral radius of the matrix $M$. This equality also shows that the limit do exists and that 
\[
C_\infty(a)=\frac{-1}{\pi}\max\left\{\ln\left(\rho\left(G^+_{2\pi}(0,1)\right)\right);\ln\left(\rho\left(G^+_{2\pi}(0,-1)\right)\right)\right\}\pp
\]
In other words the problem of finding $C_\infty$ is simply reduced to the analysis of two spectral radii. In fact it can be proved that $G^+_{2\pi}(0,1)=G^+_{2\pi}(0,-1)^*$ so there is really only one spectral radius here. To prove this equality it suffices to remark that $G^+_s(x_0,\xi_0)$ and $G_s^+(x_s,-\xi_s)^*$ satisfy the same differential equation. Equivalently, it is easy to prove this equality when $a$ is piecewise constant and by an argument of density the result is also true for every smooth function $a$. Notice that when $n=1$ the two matrices $G_{2\pi}^+(0,1)$ and $G_{2\pi}^+(0,-1)$ are equal but this is not true in the general case since $G^+$ need not be Hermitian. In conclusion we proved that
\begin{equation}\label{Cinftytoymodel}
C_\infty(a)=\frac{-1}{\pi}\ln\left(\rho\left(G^+_{2\pi}(0,1)\right)\right)\pp
\end{equation}

\paragraph*{}We are only going to deal with a particular case of damping terms but it will be general enough to exhibit all the behaviors we want. Take $A_1$, $A_2$ and $A_3$  three positive definite hermitian matrices with their eigenvalues in $(0;1]$, we know there exists three matrices $a_1$, $a_2$ and $a_3$ also definite positives such that $\exp(-a_j)=A_j$. Now take $\psi$ a smooth, non negative cut-off function such that $\int_{S^1}\psi \dd \lambda =1$ and $\supp \psi \subset (0;2\pi/3)$. The damping terms we are interested in are of the form 
\begin{equation}\label{eq:apiecewiseconstant}
a(x)=a_1\psi(x)+a_2\psi(x+2\pi/3)+a_3\psi(x+4\pi/3)
\end{equation} 
and with this condition we simply have $G_{2\pi}^+(0,1)=A_1A_2A_3$ and $G_{2\pi}^+(0,-1)=A_3A_2A_1=G_{2\pi}^+(0,1)^*$. Let us compare $C_\infty(a)$ and $C_\infty(2a)$, according to \eqref{Cinftytoymodel} we have
\[
C_\infty(a)=\frac{-1}{\pi}\ln\left(\rho(A_1A_2A_3)\right) \;\text{ and }\; C_\infty(2a)=\frac{-1}{\pi}\ln\left(\rho(A_1^2A_2^2A_3^2)\right)\pp
\]
If we use a program to randomly generate the $A_j$ it is not hard to find some function $a$ such that $C_\infty(2a)>2C_\infty(a)$ for example with 
\[
A_1=\begin{pmatrix}
0.87     &       0.21 + 0.09i  \\
    0.21 - 0.09i  &  0.51 
\end{pmatrix}\vv \; A_2=\begin{pmatrix}
0.35     &     - 0.23 + 0.08i  \\
  - 0.23 - 0.08i   & 0.61 
\end{pmatrix}\]\[ \text{ and }\, A_3=\begin{pmatrix}
 0.23      &      0.11 - 0.21i  \\
    0.11 + 0.21i  &  0.25 
\end{pmatrix}\pp
\]
It is even possible to have $C_\infty(2a)<C_\infty(a)$, for example with 
\[
A_1=\begin{pmatrix}
0.49   &  0.46 - 0.11i \\ 
    0.46 + 0.11i   & 0.52 
\end{pmatrix}\vv \; A_2=\begin{pmatrix}
0.49     &    - 0.02 + 0.3i  \\
  - 0.02 - 0.3i  &  0.58       
\end{pmatrix}     
\]
\[\text{ and }\, A_3=\begin{pmatrix}
 0.52      &   - 0.3 - 0.33i  \\
  - 0.3 + 0.33i  &  0.37    
\end{pmatrix}\pp\]
\begin{figure}[h!] 
	\centering
		\includegraphics[width=0.91\textwidth]{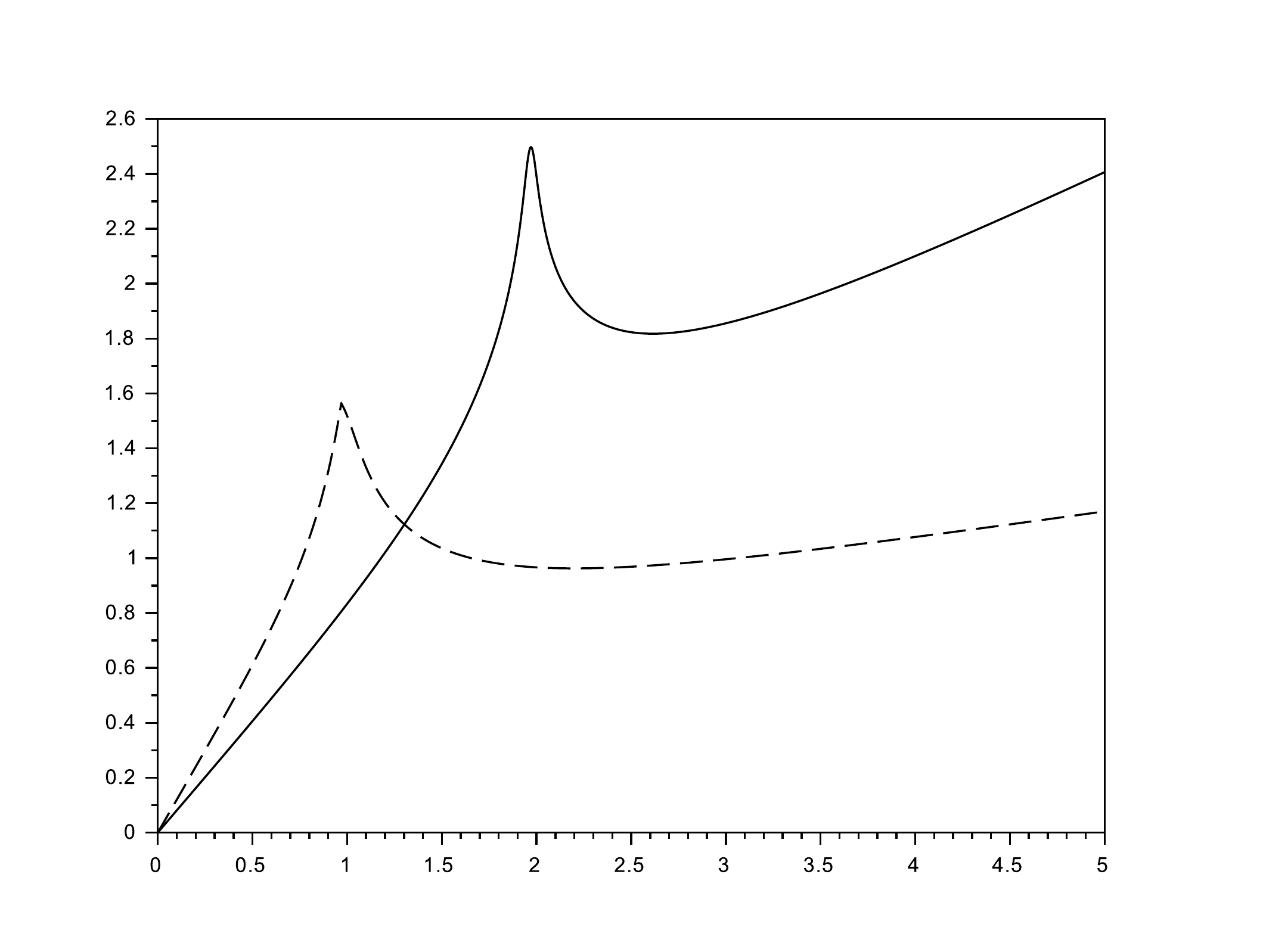}
	\caption{Graphs of $\lambda\mapsto C_\infty(\lambda a)$ for the first example (full line) and the second (dotted line). }
	
\end{figure}
This proves that $C_\infty$ is neither monotonous nor positively homogeneous. Note that even with $A_i\in \mathscr M_n(\mathbf R)$ there are examples of damping terms $a$ such that $C_\infty(2a)<C_\infty(a)$ or $C_\infty(2a)>2C_\infty(a)$. Figure 2 shows the behavior of $\lambda\mapsto C_\infty (\lambda a)$ for the two previous examples. 

We are going to use the same method to study the additivity of $C_\infty$. Assume now that $\supp \psi \subset (0;\pi/2)$, we look at two damping terms defined by 
\[
a(x)=a_1\psi(x)+a_2\psi(x+\pi)\; \text{ and } \; b(x)=b_1\psi(x+\pi/2)+b_2\psi(x+3\pi/2)\pp
\]
By equality \eqref{Cinftytoymodel} we get
\[
C_\infty(a+b)=\frac{-1}{\pi}\ln\left(\rho(A_1B_1A_2B_2)\right)\, \text{ and }\, C_\infty(a)+C_\infty(b)=\frac{-1}{\pi}\left(\ln\left(\rho(A_1A_2)\right)+\ln\left(\rho(B_1B_2)\right)\right)\pp
\]
Then again, using a program to randomly generate the $A_j$ and the $B_j$ it's not hard to find $a$ and $b$ such that $C_\infty(a+b)<C_\infty(a)+C_\infty(b)$, for example with 
\[
A_1=\begin{pmatrix}
0.27     &     - 0.15 - 0.15i  \\
  - 0.15 + 0.15i  &  0.18   
\end{pmatrix}\vv \; A_2=\begin{pmatrix}
   0.31      &     0.25 + 0.3i  \\
    0.25 - 0.3i   & 0.54    
\end{pmatrix}\vv
\]
\[
B_1=\begin{pmatrix}
   0.65      &      0.35 - 0.28i  \\
    0.35 + 0.28i  &  0.38  
\end{pmatrix} \;\text{ and }\; B_2=\begin{pmatrix}
   0.05       &   - 0.04 + 0.05i  \\
  - 0.04 - 0.05i  &  0.08      
\end{pmatrix}
\]
we find $C_\infty (a+b) \approx 1.45$ and $C_\infty(a)+C_\infty(b)\approx 2.99$. Conversely, it is possible to have $C_\infty(a+b)>C_\infty(a)+C_\infty(b)$, for example with 
\[A_1=\begin{pmatrix}
  0.17     &       0.07 - 0.11i  \\
    0.07 + 0.11i   & 0.12        
\end{pmatrix}\vv\; A_2=\begin{pmatrix}
    0.32      &    - 0.09 - 0.35i  \\
  - 0.09 + 0.35i  &  0.61    
\end{pmatrix} \vv\]
\[
 B_1=\begin{pmatrix}
    0.13     &     - 0.19 + 0.04i  \\
  - 0.19 - 0.04i   &  0.4  
\end{pmatrix} \;\text{ and } \; B_2=\begin{pmatrix}
0.18     &       0.01 + 0.13i  \\
    0.01 - 0.13i  &  0.23  
\end{pmatrix}
\]
we find $C_\infty (a+b) \approx 1.87$ and $C_\infty(a)+C_\infty(b)\approx 1.20$.
\paragraph*{}However $C_\infty$ still has some kind of homogeneous behavior as $\lambda$ tends to infinity. Assume for example that $a$ is a piecewise constant function (not necessarily continuous) or that $a$ is of the form \eqref{eq:apiecewiseconstant} but with arbitrarily many $a_i$ instead of only $3$. In this case there exists some positive definite Hermitian matrices $A_i$ with eigenvalues in $(0;1]$ such that 
\[
C_\infty(a)=\frac{-1}{\pi}\ln\left(\rho(A_1\ldots A_j)\right)
\] 
and such that for every real $\lambda\geq 0 $ we have
\[
C_\infty(\lambda a)=\frac{-1}{\pi}\ln\left(\rho(A_1^\lambda \ldots A_j^\lambda)\right)\pp
\]
We are going to prove that in this case $\lim_{\lambda\to\infty} C_\infty(\lambda a)/\lambda $ exists, is non-negative and finite. The first thing to note is that every $A_i^\lambda$ converges to some orthogonal projector $P_i$ so $A_1^\lambda \ldots A_j^\lambda$ converges to $P_1 \ldots P_j$, which has a spectral radius of either $0$ or not. If $\rho(P_1\ldots P_j)=r\neq 0$ then $\rho(A_1^\lambda \ldots A_j^\lambda)$ also converges to $r$ and thus $C_\infty(\lambda a)/\lambda$ converges to $0$. We may thus assume from now on that the spectral radius of $P_1\ldots P_j $ is $0$. Remark that each coefficient of $A_i^\lambda$ is a polynomial in the eigenvalues of $A_i^\lambda$. Let us call $P_\lambda=X^n+\sum_{i=0}^{n-1} b_i(\lambda)X^i $ the characteristic polynomial of $A_1^\lambda\ldots A_j^\lambda$, since the determinant is also a polynomial we get that each coefficient of $P_\lambda$ is a polynomial in the eigenvalues of the matrices $A_i^\lambda$. If $\xi$ is an eigenvalue of $A_i$ then $\xi^\lambda$ is an eigenvalue of $A_i^\lambda$ and so each of the coefficients of $P_\lambda$ can be written as 
\begin{figure}
	\centering
		\includegraphics[width=0.91\textwidth]{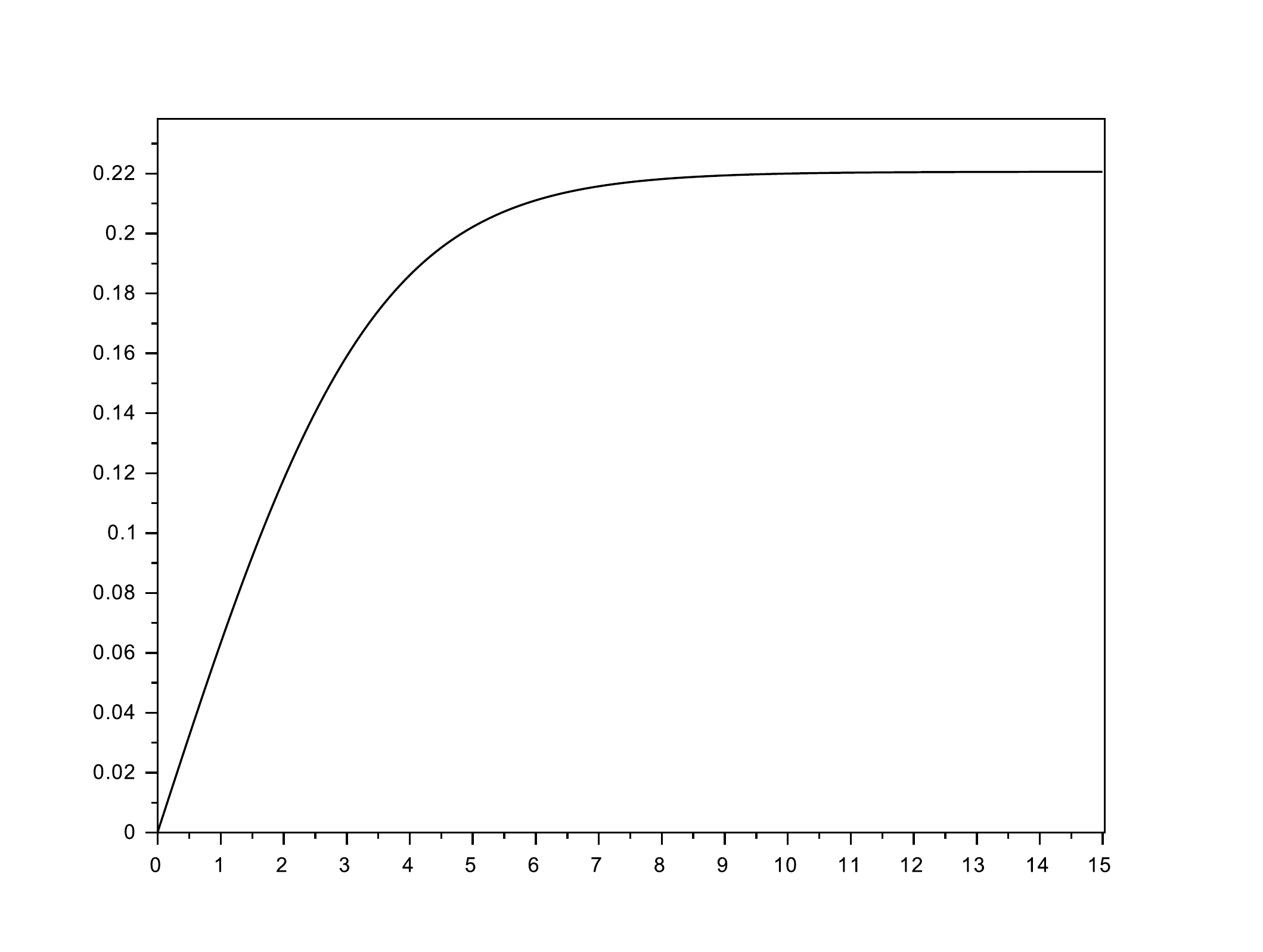}
	\caption{Graph of $\lambda\mapsto C_\infty(\lambda a)$ for some damping term $a$ with $\lim_\lambda C_\infty(\lambda a)/\lambda=0$. }
	
\end{figure}
\[
b_i(\lambda)=\sum_{j=0}^{k_i} c_{i,j} \beta_{i,j}^\lambda \; \text{ with } c_{i,j}\in \mathbf C^* \text{  and } \beta_{i,0}>\beta_{i,1}>\cdots \beta_{i,k} > 0\pp
\]
Since $\rho(A_1^\lambda \ldots A_j^\lambda)$ converges to $0$ we know that $P_\lambda$ converges to $X^n$ and so every $\beta_{i,j}$ must be in in $(0;1)$. Now look at the polynomial $\widehat{P}_\lambda(X)=\gamma^{\lambda n}P(X/\gamma^\lambda)$, we have
\[
\widehat{P}_{\lambda}(X)=X^n+\sum_{i=0}^{n-1}\gamma^{\lambda(n-i)}b_i(\lambda)X^i\; \text{ and}
\]
\[
\gamma^{\lambda(n-i)}b_i(\lambda)=\left(\gamma^{n-i}\beta_{i,0}\right)^\lambda \left(c_{i,0}+\sum_{j=1}^{k_i}  c_{i,j}\left( \frac{\beta_{i,j}}{\beta_{0,j}}\right)^\lambda\right)\pp
\]
  For this reason there exists a unique\footnote{$\gamma=\min_i \beta_{i,0}^{1/(i-n)}$} real number $\gamma>1$ such that $\widehat{P}_\lambda(X)=\gamma^{\lambda n}P(X/\gamma^\lambda)$ converges to some unitary polynomial $Q\neq X^n$. This means that the roots of $\widehat{P}_\lambda$ converge to the roots of $Q$. Let $\xi$ be a root of $Q$ with maximal modulus, recall that $\xi\neq 0$ because $Q\neq X^n$. A complex number $z$ is a root of $P_\lambda$ if and only if $\gamma^\lambda z$ is a root of $\widehat{P}_\lambda$ and these roots are converging to the ones of $Q$. We deduce from this that $\gamma^\lambda\rho(A_1^\lambda \ldots A_j^\lambda)$ converges to $|\xi|$ and we finally have 
\begin{equation}\label{eq:limitecinftyagain}
\lim_{\lambda\to\infty}\frac{C_\infty(\lambda a)}{\lambda}=\frac{-1}{\pi}\ln(\gamma^{-1})
\end{equation}
which is exactly what we wanted.  The very same kind of argument also shows that \[\lim_{\lambda \to 0^+} \frac{C_\infty(\lambda a)}{\lambda}\] exists and is finite. Numerical simulations seems to indicate that we always have
\[
\lim_{\lambda\to \infty} \frac{C_\infty(\lambda a)}{\lambda} \leq \lim_{\lambda\to 0^+} \frac{C_\infty(\lambda a)}{\lambda}
\] 
but the function $\lambda \mapsto C_\infty(\lambda a)/\lambda$ needs not to be monotonous as shown with Figure 4.

\begin{figure}
	\centering
		\includegraphics[width=1\textwidth]{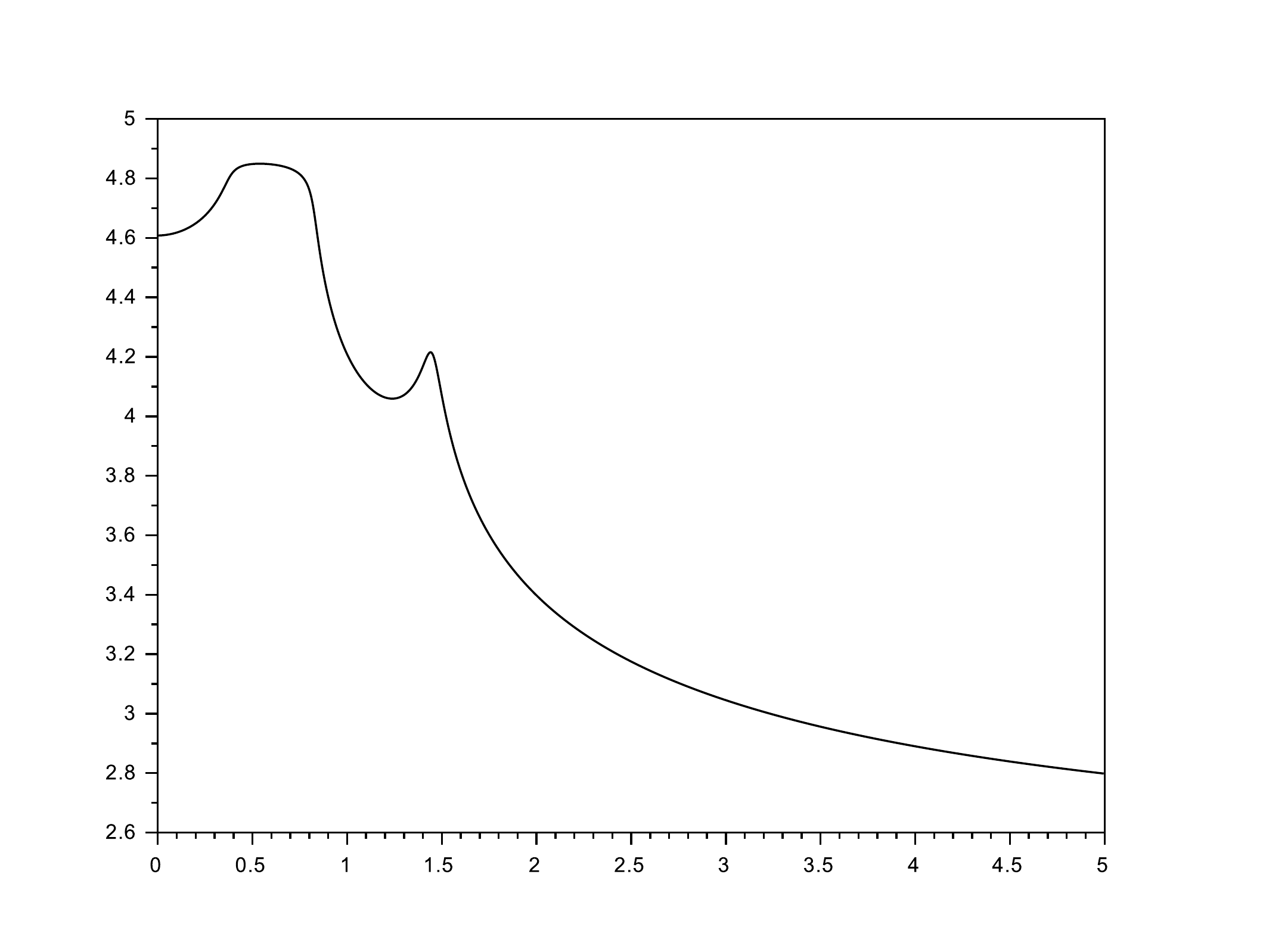}
	\caption{Graph of $\lambda\mapsto C_\infty (\lambda a)/\lambda$ for some damping term $a$ of the form \eqref{eq:apiecewiseconstant} but with five $a_i$ instead of three. Here $\lim_\lambda C_\infty(\lambda a)/\lambda$ is positive.}
	\label{fig:Cinfty}
\end{figure}

\paragraph*{} A very natural thing to do is to ask oneself if property \eqref{eq:limitecinftyagain} is still true in a more general setting, that is, is it still true with any smooth $a$ on a  general manifold ? Unfortunately several difficulties prevent us to answer this question. For example notice that on a general manifold there is no equivalent of the formula \eqref{Cinftytoymodel} and that it is not even clear that $\|a_k-a\|_\infty\to 0$ implies $C_\infty(a_k)\to C_\infty(a)$ on a general manifold. Even on the circle where this is true it does not mean that 
\[
\lim_{k\to \infty} \lim_{\lambda\to \infty} \frac{C_\infty(\lambda a_k)}{\lambda}=\lim_{\lambda\to \infty} \frac{C_\infty(\lambda a)}{\lambda}
\]
and so it is not clear that $\lim_{\lambda\to\infty} C_\infty(\lambda a)/\lambda $ exists for a smooth $a$ even in the simple case of the circle.

\end{document}